\documentclass[12pt]{amsart}
\usepackage{amssymb}

\newtheorem{theorem}{Theorem}[section]

\theoremstyle{definition}

\newtheorem{remark}[theorem]{Remark}

\numberwithin{equation}{section}

\begin{document}

\baselineskip=15.5pt

\title[Diagonal property of symmetric product of a curve]{Diagonal property of
the symmetric product of a smooth curve}

\author[I. Biswas]{Indranil Biswas}

\address{School of Mathematics, Tata Institute of Fundamental
Research, Homi Bhabha Road, Bombay 400005, India}

\email{indranil@math.tifr.res.in}

\author[S. K. Singh]{Sanjay Kumar Singh}

\address{Institute of Mathematics, Polish Academy of Sciences, Warsaw,00656, Poland}

\email{s.singh@impan.pl}



\thanks{The first named author is supported by the J. C. Bose Fellowship. The
second named author is supported by IMPAN Postdoctoral Research Fellowship.}

\date{}

\begin{abstract}
Let $C$ be an irreducible smooth projective curve defined over an algebraically closed field.
We prove that the symmetric product ${\rm Sym}^d(C)$ has the diagonal property for all
$d\, \geq\, 1$. For any positive integers $n$ and $r$, let ${\mathcal Q}_{{\mathcal
O}^{\oplus n}_C}(nr)$ be the Quot scheme parametrizing all the torsion quotients of
${\mathcal O}^{\oplus n}_C$ of degree $nr$. We prove that ${\mathcal Q}_{{\mathcal
O}^{\oplus n}_C}(nr)$ has the weak point property.
\end{abstract}

\maketitle

\section{introduction}

In \cite{PSP}, Pragacz, Srinivas and Pati introduced the diagonal and (weak) point
properties of a variety, which we recall.

Let $X$ be a variety of dimension $d$ over an algebraically closed field $k$. It is said to
have the \textit{diagonal property} if there is a vector bundle $E\, \longrightarrow\,
X\times X$ of rank $d$, and a section $s\, \in\, H^0(X\times X,\, E)$, such that the zero
scheme of $s$ is the diagonal in $X\times X$. The variety $X$ is said to have the
\textit{weak point
property} if there is a vector bundle $F$ on $X$ of rank $d$, and a
section $t\, \in\, H^0(X,\, F)$, such that the zero scheme of $t$ is a (reduced) point
of $X$. The diagonal property implies the weak point property because the restriction of
the above section $s$ to $X\times \{x_0\}$ vanishes exactly on $x_0$.
 
These properties were extensively studied in \cite{PSP} and \cite{De}. In particular,
it was shown that
\begin{itemize}
\item they impose strong conditions on the variety,

\item on the other hand there are many example of varieties with these properties.
\end{itemize}

Here we investigate these conditions for some varieties associated to a smooth
projective curve.

Let $C$ be an irreducible smooth projective curve over $k$. For any positive integer $d$,
let $\text{Sym}^d(C)$ be the quotient of $C^d$ for the natural action of the group
of permutations of $\{1\, , \cdots\, ,d\}$. It is a smooth projective variety of dimension
$d$. We prove the following (Theorem \ref{thm2}):

\begin{theorem}\label{Thmi1}
The variety ${\rm Sym}^d(C)$ has the diagonal property.
\end{theorem}

Theorem 1 in \cite[p. 1236]{PSP} contains several examples of surfaces satisfying the
diagonal property. We note that the surface ${\rm Sym}^2(C)$ is not among them.

For positive integers $n$ and $d$, let ${\mathcal Q}_{{\mathcal
O}^{\oplus n}_C}(d)$ be the Quot scheme parametrizing the torsion quotients of
${\mathcal O}^{\oplus n}_C$ of degree $d$. Quot schemes were constructed
in \cite{Gr} (see \cite{Ni} for
an exposition on \cite{Gr}). The variety ${\mathcal Q}_{{\mathcal O}^{\oplus n}_C}(d)$
is smooth projective, and its dimension is $nd$. Note that ${\mathcal Q}_{{\mathcal
O}_C}(d)\,=\, \text{Sym}^d(C)$. These varieties ${\mathcal Q}_{{\mathcal
O}^{\oplus n}_C}(d)$ are extensive studied in algebraic
geometry and mathematical physics (see \cite{bgl}, \cite{bdw}, \cite{Ba}, \cite{br} and
references therein).

We prove the following (Theorem \ref{Quot}):

\begin{theorem}\label{thmi2}
If $d$ is a multiple of $n$, then the
variety ${\mathcal Q}_{{\mathcal O}^{\oplus n}_C}(d)$ has the weak point property.
\end{theorem}

\section{Quot Scheme and the weak point property}

We continue with the notation of the introduction.

For a locally free coherent sheaf $E$ of rank $n$ on $C$, let $\mathcal{Q}_E(d)$
be the Quot scheme parametrizing all torsion quotients of $E$ of 
degree $d$. Equivalently, $\mathcal{Q}_E(d)$ parametrizes all
coherent subsheaves of $E$ of rank $n$ and degree $\text{degree}(E)-d$. Note that
any coherent subsheaf of $E$ is locally free because any torsionfree coherent
sheaf on a smooth curve is locally free. This
$\mathcal{Q}_E(d)$ is an irreducible smooth projective variety of dimension
$nd$.

There is a natural morphism
$$
\varphi'\, :\, \mathcal{Q}_E(d)\,\longrightarrow\, \mathcal{Q}_{\wedge^n E}(d)
$$
that sends any subsheaf $S\, \subset\, E$ of rank $n$ and degree $\text{degree}(E)-d$
to the subsheaf $\bigwedge^n S \,\subset\, \bigwedge^n E$. Next note that $\mathcal{Q}_{\wedge^n
E}(d)$ is identified with the symmetric product $\text{Sym}^d(C)$ by sending any subsheaf
$S' \subset\, \bigwedge^n E$ to the scheme theoretic support of the quotient
sheaf $(\bigwedge^n E)/S'$. Let 
\begin{equation}\label{vp}
\varphi\, :\, \mathcal{Q}_E(d)\,\longrightarrow\, \text{Sym}^d(C)
\end{equation}
be the composition of $\varphi'$ with this identification of $\mathcal{Q}_{\wedge^n E}(d)$
with $\text{Sym}^d(C)$. It should be mentioned that for a subsheaf $S\, \subset\, E$ of
rank $n$ and $\text{degree}(E)-d$, the image $\varphi(S)\,\in\, \text{Sym}^d(C)$ does
not, in general, coincide with the scheme theoretic support of the quotient sheaf $E/S$.

The symmetric product ${\rm Sym}^d(C)$ is the moduli space of effective divisors of
degree $d$ on $C$. Let
\begin{equation}\label{f-1}
D\, \subset\, Y\, :=\, C\times {\rm Sym}^d(C)
\end{equation}
be the universal divisor. So the fiber of
$D$ over a point $a\, \in\, {\rm Sym}^d(C)$ is the zero dimensional
subscheme of $C$ of length $d$ defined by $a$. Let
\begin{equation}\label{f-f}
\mathcal{D}\,=\, (\text{Id}_C\times\varphi)^{-1}(D)\, \subset\, C\times
\mathcal{Q}_E(d)
\end{equation}
be the inverse image of $D$, where $\varphi$ is constructed in \eqref{vp}.

\begin{remark}\label{cong}
Let L be a line bundle on $C$. For $E$ as above, if $S\, \subset\, E$ is a subsheaf
of rank $n$ and degree $\text{degree}(E)-d$, then
$$
S\otimes L\, \subset\, E\otimes L
$$
is a subsheaf of rank $n$ and degree
$\text{degree}(E\otimes L)-d$. Therefore, we get an isomorphism
$$
\mathcal{Q}_E(d)\, \stackrel{\sim}{\longrightarrow}\,
\mathcal{Q}_{E\otimes L}(d)
$$
by sending any subsheaf $S\,\subset\, E$ to the
subsheaf $S\otimes L\,\subset\, E\otimes L$.
\end{remark}

\begin{theorem}\label{Quot}
For positive integers $d,n$ such that $d$ is a multiple of $n$, the Quot scheme
$\mathcal{Q}_{\mathcal{O}_C^n}(d)$ satisfies the weak point property.
\end{theorem}

\begin{proof}
Let $r\, \in \, \mathbb N$ be such that $d\,=\, rn$. 
Fix a closed point $x_0$ in $C$. The line bundle ${\mathcal O}_C(rx_0)$ on
$C$ will be denoted by $L$. By Remark \ref{cong} it is enough to prove the
weak point property for $\mathcal{Q}_{L^{\oplus n}}(d)$.
 
Let $\mathcal{D}\,\hookrightarrow\, C\times \mathcal{Q}_{L^{\oplus n}}(d)$ be
the divisor constructed in \eqref{f-f}. Let
\begin{equation}\label{e0}
p\, :\, \mathcal{D}\, \longrightarrow\, C~\ \text{ and }~\ \,
q:\, \mathcal{D}\, \longrightarrow\,\mathcal{Q}_{L^{\oplus n}}(d)
\end{equation}
be the projections. Taking the direct sum of copies of the natural inclusion
$$
\iota\, :\, \mathcal{O}_C\, \hookrightarrow\, \mathcal{O}_C(rx_0)\, ,
$$
we get a short
exact sequence of sheaves on $C$
\begin{equation}\label{e1}
0\, \longrightarrow\, \mathcal{O}_C^{\oplus n} \stackrel{\iota^{\oplus n}}{\longrightarrow}\,
\mathcal{O}_C(rx_0)^{\oplus n}\, \longrightarrow\, T\, \longrightarrow\, 0\, ,
\end{equation}
where $T$ is a torsion sheaf on $C$ of degree $nr\,=\, d$. Therefore, this quotient
$T$ is represented by a point of $\mathcal{Q}_{L^{\oplus n}}(d)$. Let
\begin{equation}\label{e3}
t_0\, \in\, \mathcal{Q}_{L^{\oplus n}}(d)
\end{equation}
be the point representing $T$.

The direct image
$$
F\, :=\, q_*p^* L^{\oplus n}\, \longrightarrow\,
\mathcal{Q}_{L^{\oplus n}}(d)
$$
is a vector bundle of rank $nd$, where $p$ and $q$ are the projections in
\eqref{e0}. We will construct a section of $F$.
The section of $L\,=\, \mathcal{O}_C(rx_0)$ given by the constant function $1$
will be denoted by $s_0$. Consider the section
$$
s\,:=\, \iota^{\oplus n}(s^{\oplus n}_0)\, \in\,
H^0(C,\, L^{\oplus n})\, ,
$$
where $\iota^{\oplus n}$ is the homomorphism in \eqref{e1}. We have
\begin{equation}\label{e2}
\widetilde{s}\, :=\, q_*p^*s \, \in\, H^0(\mathcal{Q}_{L^{\oplus n}}(d),\, F)\, .
\end{equation}

For the point $t_0$ in \eqref{e3}, the scheme theoretic inverse image
$$
q^{-1}(t_0)\,\subset\,
\mathcal{D}\, \subset\, C\times \mathcal{Q}_{L^{\oplus n}}(d)
$$
is $(rx_0)\times t_0$, where $q$ is the projection in \eqref{e0}. Since the section $s_0$
vanishes exactly on $rx_0$, this implies that
the section $\widetilde{s}$ in \eqref{e2} vanishes exactly on the reduced point $t_0$.
Therefore, $\mathcal{Q}_{L^{\oplus n}}(d)$ has the weak point property.
\end{proof}

\section{Diagonal property for symmetric product of curves}

\begin{theorem}\label{thm2}
For any $d\, \geq\, 1$, the symmetric product ${\rm Sym}^d(C)$ of a smooth projective
curve $C$ has the diagonal property.
\end{theorem}

\begin{proof}
Consider the divisor $D$ in \eqref{f-1}. Let
\begin{equation}\label{f0}
L\,=\, {\mathcal O}_Y(D)\, \longrightarrow\, Y
\end{equation}
be the line bundle. Now consider $Z \,:=\, Y\times {\rm Sym}^d(C)\,=\,
C\times {\rm Sym}^d(C)\times {\rm Sym}^d(C)$.
Let
\begin{equation}\label{f1}
\alpha\, :\, Z\, \longrightarrow\, C\, , \ \beta\, :\, Z\, \longrightarrow\, {\rm Sym}^d(C)~\
\text{ and } ~\ \gamma\, :\, Z\, \longrightarrow\,{\rm Sym}^d(C)
\end{equation}
be the projections defined by $(x\, ,y\, ,z)\,\longmapsto\, x$,
$(x\, ,y\, ,z)\,\longmapsto\, y$ and $(x\, ,y\, ,z)\,\longmapsto\, z$ respectively.
Let
\begin{equation}\label{f2}
\widetilde{D}\, :=\, (\alpha\times\gamma)^{-1}(D)\,\subset\,
C\times {\rm Sym}^d(C)\times {\rm Sym}^d(C) \,=\, Z
\end{equation}
be the inverse image, where $D$ is defined in \eqref{f-1}.

Let
$$
p\, :\, \widetilde{D}\, \longrightarrow\, {\rm Sym}^d(C)\times {\rm Sym}^d(C)
$$
be the projection defined by $b\, \longmapsto\, (\beta(b)\, ,\gamma(b))$, where
$\beta$ and $\gamma$ are defined in \eqref{f1}, and $\widetilde{D}$ is constructed
in \eqref{f2}. Consider the direct image
\begin{equation}\label{f4}
V\, :=\, p_*(((\alpha\times\beta)^*L)\vert_{\widetilde{D}})\, \longrightarrow\,
{\rm Sym}^d(C)\times {\rm Sym}^d(C)\, ,
\end{equation}
where $L$ is the line bundle in \eqref{f0}. The natural projection 
$$
D\, \longrightarrow\, {\rm Sym}^d(C)\, ,\  \  (x\, ,y)\,\longmapsto\, y\, ,
$$
where $D$ is defined in \eqref{f-1},
is a finite morphism of degree $d$. This implies that $p$ is a finite morphism of degree $d$.
Consequently, the direct image $V$ is a vector bundle on ${\rm Sym}^d(C)\times
{\rm Sym}^d(C)$ of rank $d$.

Consider the natural inclusion ${\mathcal O}_Y\,\hookrightarrow\, {\mathcal O}_Y(D)\,=\, L$
(see \eqref{f0}). Let
\begin{equation}\label{f-3}
\sigma_0\, \in\, H^0(Y,\, L)
\end{equation}
be the section given by the constant function $1$ using this inclusion. Let
$$
\sigma\, :=\, p_*(((\alpha\times\beta)^*\sigma_0)\vert_{\widetilde{D}})\,\in\,
H^0({\rm Sym}^d(C)\times {\rm Sym}^d(C),\, V)
$$
be the section of $V$ (constructed in \eqref{f4}) given by $\sigma_0$.

We will show that the scheme theoretic inverse image $$\sigma^{-1}(0)\, \subset\,
{\rm Sym}^d(C)\times {\rm Sym}^d(C)$$ is the diagonal.

Take any point $(a\, ,b)\,\in\, {\rm Sym}^d(C)\times {\rm Sym}^d(C)$ such that $a\,\not=\,
b$. Then there is a point $z\, \in\, C$ such that the multiplicity of $z$ in $a$
is strictly smaller than the multiplicity of $z$ in $b$. We note that the scheme theoretic
inverse image
$$
p^{-1}((a\, ,b))\, \subset\, \widetilde{D}\, \subset\, Z\,=\, C\times
{\rm Sym}^d(C)\times {\rm Sym}^d(C)
$$
is $\{(a\, ,b)\}\times \widehat{b}$, where $\widehat{b}$ is the zero dimensional
subscheme of $C$ of length $d$ defined by $b$. On the other hand, for the
section $\sigma_0$ in \eqref{f-3}, the intersection
$\sigma_0^{-1}(0)\bigcap (C\times \{a\})$ is the zero dimensional
subscheme $\widehat{a}$ of $C$ of length $d$ defined by $a$.
Since the multiplicity of $z$ in $a$
is strictly smaller than the multiplicity of $z$ in $b$, we have
$$
\sigma_0((z_0\, ,b))\,\not=\, 0\, .
$$
Consequently, $\sigma((a\, ,b))\,\not=\, 0$.

Now take a point $(a\, ,a)$ on the diagonal of ${\rm Sym}^d(C)\times {\rm Sym}^d(C)$.
We have observed above that the inverse
image $$p^{-1}((a\, ,a))\, \subset\, C$$ coincides with
the intersection $\sigma_0^{-1}(0)\bigcap (C\times {a})$. This implies that
\begin{itemize}
\item $\sigma((a\, ,a))\,=\, 0$, and

\item $\sigma^{-1}(0)$ is the reduced diagonal.
\end{itemize}
Therefore, ${\rm Sym}^d(C)$ has the diagonal property.
\end{proof}

\end{document}